\documentclass[abstract=no]{scrartcl} % abstract=no turns off the "Abstract" title
\usepackage[margin=3.8cm]{geometry}

\usepackage{amsmath}
\usepackage{bbm}%\mathbbmss{1}_
\usepackage{latexsym}
\usepackage[english]{babel}
\usepackage{amsfonts}
\usepackage[T1]{fontenc}
\usepackage[utf8]{inputenc}
\usepackage{amsthm}
\usepackage{amssymb}
\usepackage{dsfont}
\usepackage{version}
\usepackage{caption}
\usepackage[backend=bibtex,style=alphabetic]{biblatex}
\usepackage[psamsfonts]{eucal}
\usepackage{stmaryrd}
\usepackage{url}
\usepackage{hyperref}
\usepackage{color}
\usepackage{graphicx}
\usepackage{blindtext}
\usepackage{tikz}
\usetikzlibrary{arrows,matrix,positioning}
\tikzstyle{dmatrix}=[matrix of math nodes,row sep=2.5em, column sep=2.5em,
text height=1.5ex, text depth=0.25ex] 
\usepackage{mathtools}
\usepackage{microtype}
\usepackage{graphicx}
\usepackage{lipsum}
\usepackage{float}
\usepackage{csquotes}
\usepackage{empheq}
\usepackage{enumitem}
\DeclareSymbolFont{epsilon}{OML}{ntxmi}{m}{it}
\DeclareMathSymbol{\epsilon}{\mathord}{epsilon}{"0F}

\theoremstyle{plain}
\newtheorem{theorem}{Theorem}[section]
\newtheorem{lemma}[theorem]{Lemma}
\newtheorem{prop}[theorem]{Proposition}

\newtheorem{prop/Def}[theorem]{Propsition/Definition}
\newtheorem{theorem/Def}[theorem]{Theorem/Definition}
\theoremstyle{definition}

\newtheorem{Def}[theorem]{Definition}
\newtheorem{rem}[theorem]{Remark}
\newtheorem{exa}[theorem]{Example}

\def \Q {{\mathbb Q}}
\def \R {{\mathbb R}}

\def \Z {{\mathbb Z}}

\def \D {{\pmb{D}}}

\def \S {{\mathbb S}}

\def \P {{\mathbb P}}

\def \T {{\mathbb T}}

\def \X {{ X_{\Sigma}}}

\def \f11 {{\frac{\log(u_v\bar{u}_v)\log(\nu_v\bar{\nu}_v)}{\log(u_v\bar{u}_v)+\log(\nu_v\bar{\nu}_v)}}}

\def \vol {{ \operatorname{vol}}}

\def \Im {{ \operatorname{Im}}}

\def \relint {{ \operatorname{relint}}}

\def \max {{ \operatorname{max}}}
\def \dom {{ \operatorname{dom}}}

\def \convhull {{ \operatorname{convhull}}}
\def \MV {{ \operatorname{MV}}}
\def \V {{ \operatorname{V}}}
\def \I {{ \operatorname{I}}}
\def \d {{ \operatorname{d}}}

\def \Ca {{ \operatorname{Ca}}}

\def \exposed {{ \operatorname{exposed}}}
\def \timess {{ \operatorname{-times}}}

\renewbibmacro{in:}{}

\usepackage{tgpagella}
\usepackage{tgheros}
\usepackage{eulervm}
\usepackage{tikz}
\usepackage{etoolbox}
\usetikzlibrary{arrows,matrix,positioning}
\tikzstyle{dmatrix}=[matrix of math nodes,row sep=2.5em, column sep=2.5em,
text height=1.5ex, text depth=0.25ex] % dmatrix = diagram matrix

\numberwithin{equation}{section}

\setparsizes{1em}{0.15\baselineskip plus .25\baselineskip}{1em plus 1fil}
\bibliography{bibliography}
\title{Canonical decomposition of a difference of convex sets}
\author{Ana Mar\'ia Botero}
\date{}
\begin{document}

\maketitle

{\small{\begin{abstract}
Let $N$ be a lattice of rank $n$ and let $M = N^{\vee}$ be its dual lattice. In this note we show that given two compact, bounded, full-dimensional convex sets $K_1 \subseteq K_2 \subseteq M_{\R} \coloneqq M \otimes_{\Z} \R$, there is a canonical convex decomposition of the difference $K_2 \setminus K_1$ and we interpret the volume of the pieces geometrically in terms of intersection numbers of toric $b$-divisors. 

\end{abstract}}}

\section{Introduction}
Convex sets have been widely and successfully used to explore the geometry of an algebraic variety using convex geometrical methods. A well known class of examples comes from the theory of toric varieties, where the combinatorics of a lattice polypte encrypts most of the geometric properties of the corresponding projective toric variety (see \cite{CLS} and \cite{ful}). More generally, Okounkov bodies (in the literature often called Newton--Okounkov bodies) are convex sets which one can attach to an algebraic variety together with some extra geometric data, e.g.~a complete flag of subvarieties. These convex sets turn out to encode also important geometric information of the variety (see \cite{OK1,OK2} and also \cite{KK, KK2, KKa} and \cite{LM} and the references therein). 

More recently, generalizing the toric situation, in \cite{botero}, convex sets are associated to so called toric $b$-divisors, which can be though of as a limit of toric divisors keeping track of birational information. Their degree is defined as a limit. There it is shown that under some
positivity assumptions toric b-divisors are integrable and that their
degree is given as the volume of a convex set. Moreover, it is shown that the
dimension of the space of global sections of a nef toric $b$-divisor
is equal to the number of lattice points in this convex set and
a Hilbert--Samuel type formula for its asymptotic growth is given. This generalizes
classical results for classical toric divisors on toric varieties. Finally, a relation between convex bodies associated to $b$-divisors and Okounkov bodies is established. We remark that the main motivation for studying toric $b$-divisors is to be able to do arithmetic intersection theory on mixed Shimura varieties of non-compact type. Indeed, it turns out that toric $b$-divisors locally encode the singularities of the invariant metric on an automorphic line bundle over a toroidal compactification of a mixed Shimura variety of non-compact type. This note is part of an overall program to develop an arithmetic intersection theory on mixed Shimura varieties of non-compact type via convex geometric methods whose starting point is \cite{botero}.

In general, of a particular interest is to compute the volume of a convex set. Aside from this being intrinsically a question of great interest, it has applications not only in the above mentioned geometric settings, but also in other mathematical fields such as in convex optimization.

For the rest of this introduction let us fix a lattice $N$ of rank $n$, its dual lattice $M = N^{\vee}$, and two compact, bounded, full-dimensional convex sets $K_1 \subseteq K_2 \subseteq M_{\R} = M \otimes_{\Z} \R$.  
The aim of this note is to show there is a canonical convex decomposition of the difference $K_2 \setminus K_1$ and to interpret geometrically the volume of the pieces in terms of intersection numbers of toric $b$-divisors. 

The outline of this note is as follows. In Section \ref{section-duality} we recall the Legendre--Fenchel duality for convex sets. Most of the definitions and statements which we will state regarding this duality can be found in \cite{ROCK}. We also refer to \cite[Chapter 2]{BPS}. 

In Section \ref{section-decomposition} we give the canonical convex decomposition of the difference $K_2 \setminus K_1$. We start by defining what it means for two faces $F_1$ and $F_2$ of $K_1$ and $K2$, respectively, to be related, denoted by $F_1 \sim F_2$. Using this relationship, we are able to show the following main result of this section, which is Theorem \ref{convexdecom} in the text.
\begin{theorem}
Let notations be as above. Then we have that
\[
\Upsilon\left(K_2 \setminus K_1\right) \coloneqq \left\{\convhull\left(F_{1}, F_{2}\right) \, \big{|} \, F_{1} \overset{\exposed}{\leq} K_1, \, F_{2} \overset{\exposed}{\leq} K_2 \, \text{ and }\, F_{1} \sim F_{2} \right\}
\]
is a convex decomposition of the difference $K_2 \setminus K_1$.
\end{theorem}
In the polyhedral case, the above canonical decomposition gives a polyhedral subdivision of the complement of two polytopes, one contained in the other. This subdivision appears in the literature (e.g. in \cite{GP}) although it is constructed using the so called pushing method. We haven't found in the literature the method we used in Theorem \ref{convexdecom} nor have we found such a canonical decomposition in the non-polyhedral case.

In Section \ref{section-measures}, we start by recalling the definition of toric $b$-divisors from \cite{botero} and, in the nef case, their connection with convex sets. We then recall the definition of the surface area measure (and a mixed version thereof) associated to a convex set (and to a collection of convex sets) for which our main reference is the survey of Schneider \cite{BM}. Finally, we relate this measure to the intersection theory of toric $b$-divisors. 

In Section \ref{section-volume}, we give a geometric interpretation of the above canonical decomposition in terms of intersection numbers of toric $b$-divisors in the case that $K_2$ is polyhedral. The main result of this section is the following, which is Theorem \ref{ii} in the text. 
\begin{theorem}
Let notations be as above and assume that $K_2$ is a polytope. Let $\D_1$ and $\D_2$ be the nef toric $b$-divisors on $X_{\Sigma_{K_2}}$ corresponding to the convex sets $K_1$ and $K_2$, respectively, where $\Sigma_{K_2}$ denotes the normal fan of $K_2$. Then we can express the difference $\D_2^n - \D_1^n$ of degrees of the nef toric $b$-divisors $\D_1$ and $\D_2$ as a finite sum of correction terms 
\[ 
\D_2^n - \D_1^n = \sum_{F \overset{\exposed}{\leq} K_2} c_F,
\]
where the sum is over all exposed faces of $K_2$. The correction terms $c_F$ are related to top intersection numbers of toric $b$-divisors by
\begin{align*}
c_F & = \sum_{i=0}^{n-1}(n-1)!\int_{\relint(\sigma_F)\cap \S^{n-1}}\left(\phi_1(u)-\phi_2(u)\right)S(\underbrace{K_1, \dotsc, K_1}_{i \timess}, \underbrace{K_2, \dotsc, K_2}_{(n-1-i) \timess},u),
\end{align*}
where $S(\cdot )$ is the mixed surface area measure associated to a collection of convex sets; and where $\phi_1$ and $\phi_2$ denote the support functions of the convex sets $K_1$ and $K_2$, respectivly.

\end{theorem}
\hspace{4mm}

\noindent\textbf{Acknowledgements} I would like to thank J.~I.~Burgos for enlightning discussions and several corrections and comments on earlier drafts of this note.
\hspace{7mm}

\section{Legendre--Fenchel duality}\label{section-duality}
Throughout this section, $N$ will denote a lattice of rank $n$ and $M = N^{\vee}$ its dual lattice. All convex sets in $M_{\R}$ are assumed to be compact and full-dimensional. We start with some definitions.

\begin{Def}
Let $K$ be a convex set in $M_{\R}$. A convex subset $F \subseteq K$ is called a \emph{face} of $K$ if, for every closed line segment $\left[m_1,m_2\right] \subseteq K$ such that $\relint\left(\left[m_1,m_2\right]\right) \cap F \neq \emptyset$, the inclusion $\left[m_1,m_2\right] \subseteq F$ holds. A non-empty subset $F \subseteq K$ is called an \emph{exposed face} of $K$ if there exists a $v \in N_{\R}$ such that
\[
F = \left\{m \in K \, \big{|} \, \langle v,m \rangle = \min_{m' \in K}\langle v,m'\rangle \right\}.
\]
\end{Def}
\begin{rem}
Every exposed face is a face. However, not every face is exposed, as can be seen in the figure \ref{exposed}. Here, the star is a non-exposed face.
\end{rem}
 \begin{figure}[ht!]
\begin{center}
\begin{tikzpicture}[scale=1.5]
\draw (0,0)--(1,0);
  \draw (0,0)--(0,1);
  \draw (1,0)--(1,1);
  \draw (1,1) arc (0:180:0.5cm);
  \draw (1,1) node[right]{Non-exposed face};
   \draw (1,1) node{$\star$};
\end{tikzpicture}
\end{center}
\caption{Example of a non-exposed face}
\label{exposed}
\end{figure}
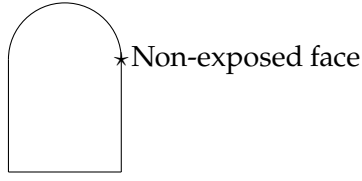

\begin{Def}
Let $\Upsilon$ be a non-empty collection of convex subsets of $M_{\R}$. $\Upsilon$ is called a \emph{convex subdivision} if the following conditions hold:
\begin{enumerate}
\item Every face of an element of $\Upsilon$ is also in $\Upsilon$.
\item\label{decomposition} Every two elements of $\Upsilon$ are either disjoint or they intersect in a common face.
\end{enumerate}
If only (\ref{decomposition}) is satisfied, then we call $\Upsilon$ a \emph{convex decomposition}. Let $\Upsilon$ be a convex subdivision or decomposition of $M_{\R}$. The \emph{support} of $\Upsilon$ is the set $|\Upsilon| \coloneqq \bigcup_{C \in \Upsilon}C$. We say $\Upsilon$ is \emph{complete} if its support is the whole of $M_{\R}$. For a given subset $E \subseteq M_{\R}$, if $|\Upsilon| = E$, we say $\Upsilon$ is a convex subdivision or decomposition of $E$.

\end{Def}
\begin{exa}
The set of all faces of a convex set $K$ is a convex subdivision of $K$. The set of all exposed faces of a convex set $K$ is a convex decomposition of $K$.
\end{exa}
A concave function is said to be \emph{closed} if it is upper semicontinuous. This includes the case of continuous, concave functions defined on compact convex sets. The \emph{support function} of a (not necessarily bounded) convex set $K$ is the function 
\[
\phi_K \colon N_{\R} \longrightarrow \underline{\R} \; (= \R \cup \{-\infty\})
\]
given by the assignment
\[
 v \longmapsto\inf_{m \in K}\langle m,v\rangle.
\] 
It is a closed, concave, conical function.

Also, recall that the \emph{Legendre--Fenchel dual} of a concave function $f\colon N_{\R} \to \underline{\R}$  is the closed, concave function 
\[
f^{\vee}\colon M_{\R} \longrightarrow \underline{\R},
\]
defined by 
\[
 m \longmapsto\inf_{v \in N_{\R}}\left(\langle m,v\rangle -f(v)\right),
\]
whose domain is the so called \emph{stability set} of $f$, which is denoted by $K_f$.

Moreover, the \emph{indicator function} of a convex set $K\subseteq M_{\R}$ is the concave function 
\[
\iota_K \colon M_{\R} \longrightarrow \underline{\R}
\]
defined by 
\[
\iota_K(m) = \begin{cases} 0 &\text{ if } m \in K, 
\\ -\infty &\text{ if } m \notin K. \end{cases}
\]

The following useful Remark can be found in \cite[Section 2.1]{BPS}.
\begin{rem}\label{gotitas} Let $K \subseteq M_{\R}$ be a convex set and let $\iota_K \colon M_{\R} \to \underline{\R}$ be its indicator function. Then we have that $\phi_K = \iota_K^{\vee}$ and $\phi_K^{\vee}= \iota_{K}$. Hence, the Legendre--Fenchel duality gives a bijective correspondence between indicator functions of compact convex sets in $M_{\R}$ and closed, concave, conical functions on $N_{\R}$.
\end{rem}

\begin{Def} Let $f$ be a concave function on $N_{\R}$. The \emph{sup-differential} $\partial f(u)$ of $f$ at $u \in N_{\R}$ is defined by 
\[
\partial f(u)\coloneqq \left\{m\in M_{\R} \, \big{|} \, \langle m, u-v\rangle \geq f(u)-f(v), \, \forall \, v \in N_{\R} \right\},
\]
if $f(u) \neq -\infty$, and $\emptyset$ if $f(u) = -\infty$. 
\end{Def}
This is a generalization to the non-smooth setting of the \emph{gradient} of a smooth function at a point. Note that in general, the sup-differential may contain more than one point. 
\begin{Def} We say that $f$ is \emph{sup-differentiable} at a point $u \in N_{\R}$ if $\partial f(u) \neq \emptyset$. The \emph{effective domain} of $f$ is the set of points where $f$ is sup-differentiable. We denote it by $\dom(\partial f)$. For a subset $V\subseteq N_{\R}$, the set $\partial f(V)$ is defined by  
\[
\partial f(V) \coloneqq \bigcup_{u\in V}\partial f(u).
\]
In particular, the \emph{image} of $\partial f$ is given by $\Im(\partial f) = \partial f(N_{\R})$. 
\end{Def}
The following propositions can be found in \cite[Section 30]{ROCK}.
\begin{prop} The sup-differential $\partial f(u)$ is a closed, convex set for all $u \in \dom(\partial f)$. It is bounded if and only if $u \in \relint(\dom(f))$. Moreover, the effective domain of $f$ is close to being convex, in the sense that 
\[
\relint(\dom(f)) \subseteq \dom(\partial f) \subseteq \dom(f).
\]
In particular, if $\dom(f) = N_{\R}$, we have $\dom(\partial f) = N_{\R}$. 
\end{prop}
\begin{prop}
If $f$ is closed, then we have that $\Im(\partial f) = \dom(\partial f^{\vee})$. Hence, the image of the sup-differential is close to being convex, in the sense that 
\[
\relint(K_f) \subseteq \Im(\partial f) \subseteq K_f.
\]

\end{prop}

\begin{Def}
Let $f$ be a closed, concave function on $N_{\R}$. We denote by $\Upsilon(f)$ the collection of all sets of the form 
\[
C_m \coloneqq \partial f^{\vee}(m) \subseteq \mathcal{P}(N_{\R}),
\]
for $m \in \dom\left(f^{\vee}\right) \subseteq M_{\R}$.
\end{Def}
The following is \cite[Proposition 2.2.8]{BPS}.
\begin{prop}
Let $f$ be a closed, concave function on $N_{\R}$. Then $\Upsilon(f)$ is a convex decomposition of $\dom(\partial f)$. In particular, if $\dom(f) = N_{\R}$, then $\Upsilon(f)$ is complete. 
\end{prop}

\begin{Def}
Let $f$ be a closed, concave function on $N_{\R}$. The \emph{Legendre--Fenchel} correspondence of $f$ 
\[
\mathcal{L}f \colon \Upsilon(f) \longrightarrow  \Upsilon(f^{\vee})
\]
is given by the assignment
\[
C \longmapsto  \bigcap_{u \in C} \partial f(u) \quad\left(=\partial f(u_0), \, \text{ for any } u_0 \in \relint(C) \right).
\]
\end{Def}
 
\begin{Def} Let $V$, $V^*$ be subsets of $N_{\R}$ and of $M_{\R}$, respectively. Moreover, let $\Upsilon$, $\Upsilon^*$ be convex decompositions of $V$ and $V^*$, respectively. We say that $\Upsilon$ and $\Upsilon^*$ are \emph{dual} convex decompositions if there exists a bijective map 
\[
\Upsilon \longrightarrow  \Upsilon^*
\]
given by the assignment 
\[
C \longmapsto  C^*
\]
and satisfying the following properties: 
\begin{enumerate}
\item For every $C, D$ in $\Upsilon$ we have that  $C \subseteq D$ if and only if $C^* \supseteq D^*$. 
\item For every $C$ in $\Upsilon$, the sets $C$ and $C^*$ are contained in orthogonal affine spaces of $N_{\R}$ and $M_{\R}$, respectively.
\end{enumerate}
\end{Def}
The following theorem is taken from \cite[Theorem 2.2.12]{BPS}. 
\begin{theorem}\label{leg_fenc_duality}
Let $f$ be a closed, concave function. Then $\mathcal{L}f$ gives a duality between $\Upsilon(f)$ and $\Upsilon(f^{\vee})$ with inverse given by $(\mathcal{L}f)^{-1} = \mathcal{L}f^{\vee}$. 
\end{theorem}
We make the following remark which can be found in \cite[Proposition 2.1.5]{convex}.
\begin{rem}\label{delta1}
Consider a convex set $K \subseteq M_{\R}$. Let $\phi_K$ be the corresponding closed, concave, conical support function and let $C \in \Upsilon(\phi_K)$. Then, for any $u \in \text{relint}( C)$ we have that $\partial \phi_K(u) \in \Upsilon(\phi_K^{\vee})$ is an exposed face of $K_{\phi_K} = K$. Conversely, every exposed face $F$ of $K$ can be obtained as $\partial \phi_K(u)$ for some $u \in N_{\R}$. Explicitly, consider $m \in \relint(F)$. Then we may take any $u \in \relint\left(\partial h^{\vee}_K(m)\right) = \relint\left(\partial \iota_K(m)\right)$. In particular, if $K$ is bounded, we get a duality between the set of exposed faces of $K$ and a convex decomposition of $N_{\R}$. 
\end{rem}
\begin{exa}\label{exa:folliation}
Let notations be as in Remark \ref{delta1} and assume that $K = P$ is a polytope. Then the Legendre--Fenchel duality gives back the classical duality between the faces of a polytope and the cones of its normal fan $\Sigma_P$.

If $K$ is not polyhedral, our convex decompositions will not be finite, as can be seen in Figure \ref{foll}. Here we have  
\begin{align*}
\phi_K(a,b) = \begin{cases}\frac{ab}{a+b}, \, &\text{ if }a,b \in \R_{\geq 0},\\
\min\{0,a,b\},\, &\text{ otherwise. } 
\end{cases}
\end{align*}
Note that here the convex decomposition of $N_{\R} \simeq \R^2$ gives us a foliation of the positive quadrant by rays.  
\end{exa}
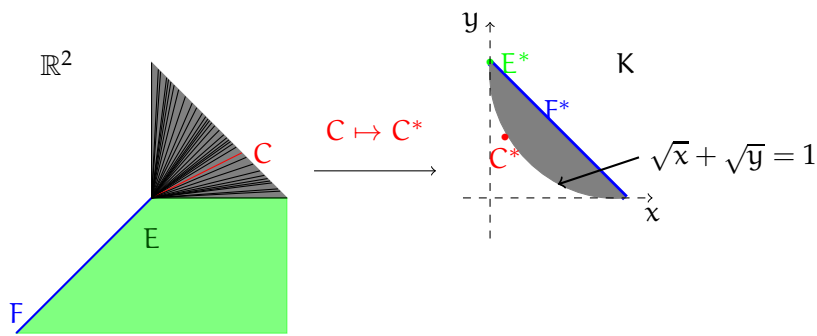
\begin{figure}[h]
\begin{center}
\begin{tikzpicture}[scale=1.8]
\filldraw[gray] (0,0)--(1,0)--(0,1)--cycle;
\filldraw[green, opacity = 0.5] (0,0)--(1,0)--(1,-1)--(-1,-1)--cycle;
\draw (0, -0.3) node[thick, black!70!green, opacity = 1.5]{{\textbf{$E$}}}; 
  \draw (0,0)--(0,1);
  \draw (0,0)--(1,0);
  \draw[thick, blue] (0,0)--(-1,-1);
  \draw(-1,-1) node[above,blue]{{$F$}}; 
  \draw (0,0)--(1/2,1/2);
  \draw (0,0)--(3/4,1/4);
\draw (0,0)--(1/4,3/4);
\draw (0,0)--(3/5,2/5);
\draw (0,0)--(2/5,3/5);
\draw (0,0)--(1/5,4/5);
\draw (0,0)--(4/5,1/5);
\draw (0,0)--(1/6,5/6);
\draw (0,0)--(5/6,1/6);
\draw (0,0)--(3/7,4/7);
\draw (0,0)--(4/7,3/7); 
\draw (0,0)--(1/7,6/7);
\draw (0,0)--(6/7,1/7);
\draw (0,0)--(5/8,3/8);
\draw (0,0)--(3/8,5/8);
\draw (0,0)--(1/8,7/8);
\draw (0,0)--(7/8,1/8);
\draw (0,0)--(4/9,5/9);
\draw (0,0)--(5/9,4/9);
\draw (0,0)--(1/9,8/9);
\draw (0,0)--(8/9,1/9);
\draw[red] (0,0)--(2/3,1/3);
\draw (0,0)--(1/3,2/3);
\draw (0,0)--(1/15,14/15);
\draw (0,0)--(14/15,1/15);
\draw (0,0)--(7/15,8/15);
\draw (0,0)--(8/15,7/15);
\draw (0,0)--(1/20,19/20);
\draw (0,0)--(19/20,1/20); 
\draw (0,0)--(7/10,3/10);
\draw (0,0)--(3/10,7/10);
\draw (0,0)--(5/12,7/12);
\draw (0,0)--(7/12,5/12);
\draw (2/3, 1/3) node[right, red]{{$C$}};
\draw[->] (1.2, 0.2) -- (2.1, 0.2);
\draw[blue,line width=0.9mm] (2.5,1)--(3.5,0);
\draw (3,0.5) node[blue, above]{{$F^*$}};
\draw (2.5,1) node[green]{\tiny{$\bullet$}};
\draw (2.5,1) node[green, right]{{$E^*$}}; 
\filldraw[gray] (2.5,1) to [bend right = 50](3.5,0) to [bend left= 0](2.5,1) --cycle;
\draw[->,dashed] (2.5,-0.3)--(2.5,1.3);
\draw[->,dashed] (2.3,0)--(3.7,0);
\draw(2.5,1.3) node[left]{$y$};
\draw(3.7,0) node[below]{$x$};
\draw (47/18, 4/9) node[red]{\tiny{$\bullet$}};
\draw (47/18, 4/9) node[below, red]{{$C^*$}};
\draw (1.65, 0.5) node[red]{$C \mapsto C^*$}; 
\draw (3.5, 1) node{$K$};
\draw (-0.7,1) node{$\R^2$};
\draw[<-, thick] (3,0.1) to (3.6,0.3);
\draw(3.6,0.3) node[right]{$\sqrt{x} + \sqrt{y} = 1$};
 
\end{tikzpicture}
\end{center}
\caption{Legendre--Fenchel correspondence in the non-polyhedral case}\label{foll}
\end{figure}
\section{Canonical decomposition of a difference of convex sets}\label{section-decomposition}
Throughout this section, $N$ will denote a lattice of rank $n$ and $M = N^{\vee}$ its dual lattice. All convex sets in $M_{\R}$ are assumed to be compact, bounded and full-dimensional.

Let $K_1 \subseteq K_2$ be two convex sets in $M_{\R}$ with corresponding support functions $\phi_{K_1}, \phi_{K_2} \colon N_{\R} \to \R$. The aim of this section is to give a canonical decomposition of the difference $K_2 \setminus K_1$.
\begin{Def}\label{complement} We define two complete convex decompositions $\Sigma_{K_1}$ and $\Sigma_{K_2}$ of $N_{\R}$ by setting
\[
\Sigma_{K_i} \coloneqq \Upsilon(\phi_{K_i})
\]
for $i = 1,2$. 
\end{Def}
Note that the elements in $\Sigma_{K_i}$ for $i = 1,2$ are cones. This follows from the fact that the convex set $C_m$ corresponding to an $m \in \relint(K_i)$ is $\{0\}$. Hence, we will call $\Sigma_{K_i}$ a fan, eventhough it may not be finite nor rational. 

It follows from Remark \ref{delta1} that the Legendre--Fenchel duality gives an order-reversing, bijective correspondence between cones in $\Sigma_{K_i}$ and the set of exposed faces of $K_i$ for $i = 1,2$. For $F \leq K_i$ an exposed face, we will denote by $\sigma_F$ the cone in $\Sigma_{K_i}$ given by this correspondence.\\
The following is a key definition for giving the canonical decomposition of the difference $K_2 \setminus K_1$.   
\begin{Def} Let $F_1 \leq K_1$ and $F_2 \leq K_2$ be exposed faces. We say that $F_1$ \emph{is related to} $F_2$ (and denote it by $F_1 \sim F_2$) if and only if
\[
\relint\left(\sigma_{F_1}\right) \cap \relint\left(\sigma_{F_2}\right) \neq \emptyset
\]
is satisfied.
\end{Def} 
\begin{Def}
Let $\Sigma \subseteq N_{\R}$ be a complete (not necessarily finite nor rational) conical subdivision of $N_{\R}$. We say that $\Sigma$ is a \emph{difference conical subdivision} for $K_1$ and $K_2$, and denote it by $\Sigma = \Sigma_{K_2 \setminus K_1}$, if the following two conditions are satisfied: 
\begin{enumerate}
\item\label{condition1} $\Sigma$ is a smooth refinement of both $\Sigma_{K_1}$ and $\Sigma_{K_2}$.
\item\label{condition2} Let $F_1 \leq K_1$ and $F_2 \leq K_2$ be exposed faces. If $F_1 \sim F_2$, then there exists a $\tau \in \Sigma(1)$ such that $\tau \in \text{relint}(\sigma_{F_1}) \cap \text{relint}(\sigma_{F_2})$.
\end{enumerate}
\end{Def}
\begin{rem}
Note that given two $n$-dimensional convex sets $K_1 \subseteq K_2$, we can always find a difference conical subdivision $\Sigma_{K_1\setminus K_2}$.
%by first including all of the faces of all of the cones of both $\Sigma_{K_1}$ and  $\Sigma_{K_2}$ and then taking a smooth refinement. 
\end{rem} 
The next theorem gives us the canonical decomposition of the difference $K_2 \setminus K_1$. It is the main result of this section. 
\begin{theorem}\label{convexdecom}
Let $K_1 \subseteq K_2$ be two $n$-dimensional convex sets in $M_{\R}$. Then we have that
\[
\Upsilon\left(K_2 \setminus K_1\right) \coloneqq \left\{\convhull\left(F_{1}, F_{2}\right) \, \big{|} \, F_{1} \overset{\exposed}{\leq} K_1, \, F_{2} \overset{\exposed}{\leq} K_2 \, \text{ and }\, F_{1} \sim F_{2} \right\}
\]
is a convex decomposition of the difference $K_2 \setminus K_1$. 
\end{theorem}
\begin{proof}
Let $F_1, F_1' \leq K_1$ and $F_2, F_2' \leq K_2$ be exposed faces such that $F_1 \sim F_1'$ and $F_2 \sim F_2'$. Now, since $F_1 \sim F_2$ we may fix a $v \in N_{\R}$ such that 
\[
F_1 = \left\{m \in K_1 \, \big{|} \, \langle v,m\rangle = \min_{m' \in K_1}\langle v,m'\rangle \right\}\quad \text{ and } \quad F_2 = \left\{m \in K_2 \, \big{|} \, \langle v,m \rangle = \min_{m' \in K_2}\langle v,m'\rangle \right\}.
\]
Analogously, we may fix $v' \in N_{\R}$ defining the faces $F_1' \sim F_2'$. Note that related faces live in parallel hyperplanes. 

Now, let us show that $\convhull\left(F_1, F_2\right) \subseteq K_2\setminus K_1$.  
Let $m \in \convhull(F_1, F_2)$. The fact that $m \in K_2$ is clear. Now, let $\lambda_1, \lambda_2$ be non-negative real numbers satisfying $\lambda_1 + \lambda_2 = 1$ and such that 
\[
m = \lambda_1m_1 + \lambda_2m_2
\]
for $m_1 \in F_1$ and $m_2 \in F_2$. 
Since $m_1 \in F_1$, $m_2 \in F_2$ and $K_1 \subseteq K_2$, we have
\[
\langle v, m_1 \rangle = \min_{m' \in K_1} \langle v,m' \rangle \geq \min_{m' \in K_2} \langle v, m' \rangle = \langle v,m_2 \rangle.
\]
Hence, we obtain
\[
\langle v,m \rangle = \lambda_1 \langle v,m_1 \rangle + \lambda_2 \langle v,m_2 \rangle \leq \lambda_1\langle v,m_1 \rangle + \lambda_2 \langle v, m_1 \rangle = \langle v, m_1 \rangle,
\]
which implies that 
\begin{align}\label{fff11}
\convhull(F_1, F_2) \cap K_1 = F_1, 
\end{align}
in particular $\convhull\left(F_1, F_2\right) \subseteq K_2\setminus K_1$.

Now, let $m' \in K_2 \setminus K_1$. Let $F_2 \leq K_2$ be the smallest exposed face such that $m' \in F_2$. Consider a $\tau \in \relint\left(\sigma_{F_2}\right)$ and let $F_1 \leq K_1$ be the unique exposed face such that $\tau \in \relint\left(\sigma_{F_1}\right).$ Then $m' \in \convhull\left(F_1, F_2\right)$. 

Hence, we have shown that
\[
K_2 \setminus K_1 = \left\{\convhull\left(F_{1}, F_{2}\right) \, \bigg{|} \, F_{1} \overset{\text{\tiny{exposed}}}{\leq} K_1, \, F_{2} \overset{\text{\tiny{exposed}}}{\leq} K_2 \, \text{ and }\, F_{1} \sim F_{2} \right\}.
\] 
It remains to show that this is indeed a convex decomposition, i.e.~that 
\begin{align}\label{intersectionplane}
\convhull\left(F_1, F_2\right) \cap \convhull\left(F_1', F_2'\right) 
\end{align}
is either empty or a face of both. If the intersection is empty, then we are done. Hence, assume that $\convhull\left(F_1, F_2\right) \cap \convhull\left(F_1', F_2'\right) \neq \emptyset$.
By (\ref{fff11}), the case in which we have that $\convhull(F_1, F_2) \subseteq \convhull(F_1', F_2')$ or in which $\convhull(F_1', F_2') \subseteq \convhull(F_1, F_2)$ is also clear. Hence, assume that 
\[
 \convhull\left(F_1, F_2\right) \setminus \convhull\left(F_1', F_2'\right) \neq \emptyset
\]
and
\[
\convhull\left(F_1', F_2'\right) \setminus \convhull\left(F_1, F_2\right) \neq \emptyset.
\]
Let $H$ be a hyperplane separating $F_1$ and $F_1'$, which exists by the definition of an exposed face. Then, since $F_1$ is parallel to $F_2$ and $F_1'$ is parallel to $F_2'$, we can choose $H$ to be a separating hyperplane of $F_2$ and $F_2'$ as well. 
  
The existence of this separating hyperplane implies that 
\[
\convhull\left(F_1, F_2\right) \cap \convhull\left(F_1', F_2'\right) = \convhull\left(F_1 \cap F_1', F_2 \cap F_2'\right)
\]
which proves that the intersection in (\ref{intersectionplane}) is a face of both. This concludes the proof of the proposition. 
\end{proof}
\begin{rem}
As was mentioned in the introduction, in the polyhedral case, the above canonical decomposition a polyhedral subdivision of the complement of two polytopes, one contained in the other. This subdivision appears in the literature (e.g. in \cite{GP}) although it is constructed using the so called pushing method. We haven't found in the literature the method we used in Theorem \ref{convexdecom} nor have we found such a canonical decomposition in the non-polyhedral case. 
\end{rem}
Let $K_1 \subseteq K_2$ be full-dimensional convex sets in $M_{\R}$. 
\begin{Def}\label{correctionterm}
Let $F \leq K_2$ be an exposed face. To $F$ we associate the \emph{correction set} 
\[
K_F \coloneqq \bigcup_{\tau \in \relint\left(\sigma_F\right)}\convhull\left(F, F_{1,\tau}\right),
\]
where for $\tau \in \relint\left(\sigma_F\right)$, the face $F_{1,\tau}$ is the unique exposed face of $K_1$ such that $\tau \in \relint\left(\sigma_{F_{1,\tau}}\right)$. The associated \emph{correction term} $c_F$ is defined as $n!$ times the volume of $K_F$, i.e.~
\[
c_F \coloneqq n!\, \vol\left(K_F\right).
\]
\end{Def}
\begin{rem} Note that by Theorem \ref{convexdecom} we have that
\[
K_2\setminus K_1 = \bigcup_{F\overset{\text{\tiny{exposed}}}{\leq} K_2} K_F
\]
is a convex decomposition and hence 
\[
n! \, \vol \left(K_2 \setminus K_1 \right) = \sum_{F\overset{\text{\tiny{exposed}}}{\leq} K_2} c_F.
\] 
\end{rem}
Let's look at a simple $2$-dimensional polyhedral example. 
\begin{exa}
Consider the simplex $K_1$ contained in the square $K_2$ as in the Figure \ref{simplexinsquare}. Here, the different colors show the correction sets associated to the faces of $K_2$. Figure \ref{fanssimplexinsquare} shows the dual picture with the fans $\Sigma_{K_2}, \, \Sigma_{K_2 \setminus K_1}, \, \Sigma_{K_1}$. 
\begin{figure}[ht!]
\begin{center}
\begin{tikzpicture}[scale=1.5]
\draw[line width=0.9mm](-2,-2)--(2,-2)--(2,2)--(-2,2)--cycle;
\draw[green!100,line width=0.9mm](2,-2)--(2,2);
\draw[line width=0.9mm](0,0)--(2,0)--(0,2)--cycle; 
\filldraw[yellow!30] (-2,-2)--(0,0)--(2,0)--(2,-2)--cycle;
\filldraw[red!30] (-2,-2)--(0,0)--(0,2)--(-2,2)--cycle;
\filldraw[blue!30] (2,2)--(0,2)--(2,0)--cycle;
\draw[yellow!90,line width=0.9mm] (-2,-2)--(2,-2);
\draw[red!90, line width=0.9mm](-2,-2)--(-2,2);
\draw (2,2) node[blue!90]{\Large{$\bullet$}};
\draw (0,0) node[above right]{$K_1$}; 
\draw (-2,2) node[below right]{$K_2$};
\draw (0,-2) node[below]{$F_2$};
\draw (2,-1) node[right]{$F_2'$};
\draw (2,2) node[right]{$P_2$};
\draw (1,0) node[below, brown!100]{$F_{1,\tau_0}$}; 
\draw (2,0) node[right, brown!90]{$F_{1, \tau_1}=F_{1,\tau_2}$};
\draw (0,2) node[above, brown!90]{$F_{1,\tau_4}$};
\draw[orange!90, line width=0.5mm] (0,0)--(2,0);
\draw[purple!90, line width=0.5mm] (0,2)--(2,0);
\draw (0.8,0.75) node[above, brown!90]{$F_{1,\tau_3}$};
\draw (2,0) node[brown!90]{$\bullet$};
\draw (0,2) node[brown!20]{$\bullet$};
\end{tikzpicture}
\end{center}
\caption{Canonical decomposition of the complement of the simplex contained in the square}\label{simplexinsquare}
\end{figure}
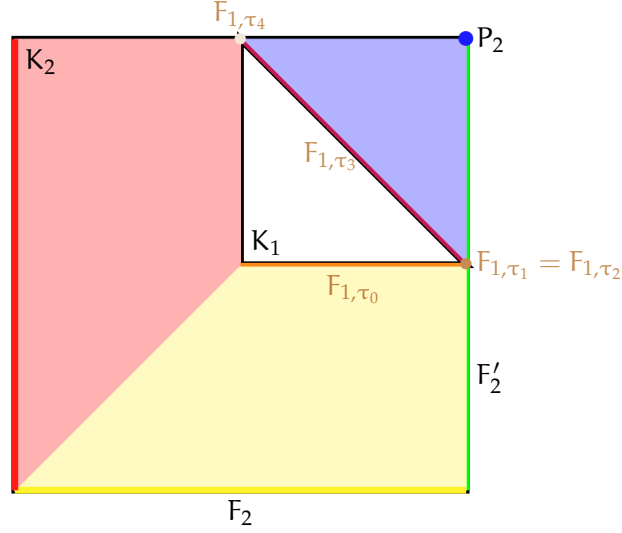

\begin{figure}[ht!]
\begin{center}
\begin{tikzpicture}[scale=1.5]
\draw[red!90, line width=0.9mm](-3,-0)--(-2,0);
\draw (-4,0)--(-2,0);
\draw (-3,-1)--(-3,1);
\draw (-4,1) node{$\Sigma_{K_2}$};
\draw (-3,0.8) node[right]{$\sigma_{F_2}$};
\draw[yellow!100,line width=0.9mm](-3,0)--(-3,1);
\draw (-3.8,0) node[above]{$\sigma_{F_2'}$};
\draw[green!100,line width=0.9mm] (-3,0)--(-4,0);
\filldraw[blue!30] (-4,0)--(-3,-1)--(-3,0)--cycle;
\draw (-3.5,-0.5) node{$P_2$};
\draw (0,0)--(2,0);
\draw (1,-1)--(1,1);
\draw (0,-1)--(1,0);
\draw (0,1) node{$\Sigma_{K_2 \setminus K_1}$};
\draw[densely dotted,magenta] (1,0)--(0,-1);
\draw (0,-1) node[above, magenta]{$\tau_3$};
\draw[densely dotted,magenta] (1,0)--(0.5,-1);
\draw (0.5,-1) node[above, magenta]{$\tau_4$};
\draw[densely dotted,magenta] (1,0)--(0,-0.5);
\draw (0,-0.5) node[above, magenta]{$\tau_2$};
\draw[densely dotted,magenta] (1,0)--(0,0);
\draw (0,0) node[above right, magenta]{$\tau_1$};
\draw[densely dotted,magenta] (1,0)--(1,1);
\draw (1,1) node[below right, magenta]{$\tau_0$};

\filldraw[brown!10] (3,-1)--(3,1)--(4,1)--(4,0)--cycle;
\filldraw[brown!30] (3,-1)--(5,-1)--(5,0)--(4,0)--cycle; 
\draw[dashed] (3,0)--(4,0);
\draw (4,0)--(5,0);
\draw (4,0)--(4,1);
\draw[dashed](4,-1)--(4,0);
\draw (3,-1)--(4,0);
\draw (3,1) node{$\Sigma_{K_1}$};
\draw (4.5, -0.5) node{$\sigma_{F_{1,\tau_4}}$};
\draw[orange!90, line width=0.9mm] (4,0)--(4,1);
\draw (4,1) node[below right]{$\sigma_{F_{1,\tau_0}}$};
\draw[purple!90, line width=0.9mm] (4,0)--(3,-1);
\draw (3,-0.85) node[above]{$\sigma_{F_{1,\tau_3}}$};
\draw (3.4,0) node[above]{$\sigma_{F_{1,\tau_1}} = \sigma_{F_{1,\tau_2}}$};
\end{tikzpicture}
\end{center}
\caption{Difference conical subdivision of the simplex contained in the square}\label{fanssimplexinsquare}
\end{figure}
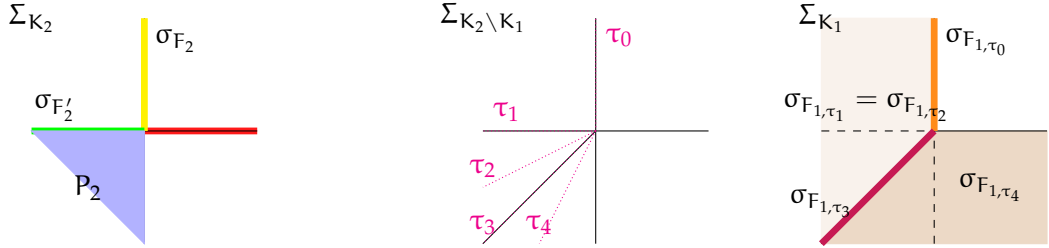
\end{exa}

\section{Toric $b$-divisors and surface area measures}\label{section-measures}

Throughout this section, $N$ will denote a lattice of rank $n$ and $M = N^{\vee}$ its dual lattice. All convex sets in $M_{\R}$ are assumed to be compact, bounded and full-dimensional. 

The goal of this section is to recall the main definitions and facts regarding toric $b$-divisors (see \cite{botero}) and to relate the intersection theory of toric $b$-divisors with the so called surface area measure (and a mixed version thereof) associated to a convex set (and to a collection of convex sets) (see \cite{BM}). 

We fix a complete, smooth fan $\Sigma \subseteq N_{\mathbb{R}} = N \otimes_{\Z} \R$ and we denote by $X_{\Sigma}$ the corresponding $n$-dimensional, complete, smooth toric variety with dense open torus $\T$. We refer to \cite{CLS} and to \cite{ful} for a more detailed introduction to toric geometry. The set $R(\Sigma)$ consists of all smooth sudivisions of $\Sigma$. This is a directed set with partial order given by $\Sigma'' \geq \Sigma'$ in $R(\Sigma)$ if and only if $\Sigma''$ is a smooth subdivision of $\Sigma'$. The \emph{toric Riemann--Zariski space} of $X_{\Sigma}$ is defined as the inverse limit 
\[
\mathfrak{X}_{\Sigma} \coloneqq \varprojlim_{\Sigma' \in R(\Sigma)}X_{\Sigma'},
\] 
with maps given by the toric proper birational morphisms $\pi_{\Sigma''}\colon X_{\Sigma''} \to X_{\Sigma'}$ induced whenever $\Sigma'' \geq \Sigma'$. The group of \emph{toric Weil $b$-divisors} on $\X$ consists of elements in the inverse limit
\[
\text{We}(\mathfrak{X}_{\Sigma})_{\mathbb{Q}} \coloneqq \varprojlim_{\Sigma' \in R(\Sigma)} \mathbb{T}\text{-}\Ca(X_{\Sigma'})_{\mathbb{Q}},
\]
where $ \mathbb{T}\text{-}\Ca(X_{\Sigma'})_{\mathbb{Q}}$ denotes the set of toric Cartier $\Q$-divisors of $X_{\Sigma'}$, with maps given by the push-forward map of toric Cartier $\Q$-divisors. We will denote $b$-divisors with bold $\D$ to distinguish them from classical divisors $D$. We can think of a toric $b$-divisor as a net of toric Cartier $\Q$-divisors $\left(D_{\Sigma'}\right)_{\Sigma' \in R(\Sigma)}$, being compatible under push-forward. 

A toric $b$-divisor $\D = \left(D_{\Sigma'}\right)_{\Sigma' \in R(\Sigma)}$ is said to be \emph{nef}, if $D_{\Sigma'} \in \mathbb{T}\text{-}\Ca(X_{\Sigma'})_{\mathbb{Q}}$ is nef for all $\Sigma'$ in a cofinal subset of $R(\Sigma)$. It follows from basic toric geometry that there is a bijective correspondence between the set of nef toric $b$-divisors and the set of $\mathbb{Q}$-valued, conical, $\Q$-concave functions on $N_{\Q}$. 

The \emph{mixed degree} $\D_1 \dotsm \D_n$ of a collection of toric $b$-divisors is defined as the limit (in the sense of nets)
\[
      \D_1\dotsm \D_n \coloneqq \lim_{\Sigma' \in R(\Sigma)} D_{1_{\Sigma'}} \dotsm D_{n_{\Sigma'}}
      \]
of top intersection numbers of toric divisors, provided this limit exists and is finite. In particular, if $\D = \D_1 = \dotsc = \D_n$, then the limit (in the sense of nets) 
\[
\D^n \coloneqq \lim_{\Sigma' \in R(\Sigma)} D_{\Sigma'}^n, 
\]
provided this limit exists and is finite, is called the \emph{degree} of the toric $b$-divisor $\D$. A toric $b$-divisor whose degree exists, is said to be \emph{integrable}. 

Let $\D_1, \dotsc,\D_n$ be a collection of toric $b$-divisors on a smooth and complete toric variety $X_{\Sigma}$ of dimension $n$ which are nef. Let $\tilde{\phi}_i \colon N_{\mathbb{Q}} \to \mathbb{Q}$ be the corresponding $\mathbb{Q}$-concave functions for $i = 1, \dotsc, n$. The \emph{mixed volume}  $\MV\left(K_1, \dotsc,K_n\right)$ of a collection of convex sets $K_1, \dotsc, K_n$ is defined by 
\begin{align}\label{def:mixedvolume}
\MV\left(K_1, \dotsc,K_n\right) \coloneqq \sum_{j=1}^n(-1)^{n-j}\sum_{1\leq i_1<\dotsb <i_j\leq n}\vol\left(K_{i_1}+ \dotsb + K_{i_j}\right), 
\end{align}
where the \enquote{$+$} refers to \emph{Minkowski addition} of convex sets.
The following theorem relates the mixed degree of nef toric $b$-divisors with the mixed volume of convex bodies. It is a combination of \cite[Theorems 4.9 and 4.12]{botero}.
    
\begin{theorem}\label{the:volume}

      With notations as above, the functions $\tilde{\phi}_i$ extend to continuous, concave functions ${\phi}_i \colon N_{\mathbb{R}} \to \mathbb{R}$.
      Moreover, their mixed degree $\D_1\dotsm \D_n$ exists, and is given by the mixed volume of the stability sets $K_{{\phi}_i}$ of the concave functions ${\phi}_i$, i.e. we have that
      \[
      \D_1\dotsm \D_n = \MV\left(K_{{\phi}_1}, \dotsc, K_{{\phi}_n}\right).
      \]
      In particular, a nef toric $b$-divisor $\D$ is integrable, and its degree is given by 
      \[
      \D^n = n!\,\vol\left(K_{{\phi}}\right),
      \]
 where ${\phi}$ is the corresponding concave function. 
 
 \end{theorem}
Now, we relate the intersection theory of toric $b$-divisors with the so called surface area measure associated to a convex set (Definition \ref{surfarea}). We start with some definitions.
\begin{Def} Consider the vector space $\R^n$ equipped with the standard euclidean metric. For  $0 \leq k \leq n$, we let $\mathcal{H}^k$ be the $k$-dimensional \emph{Hausdorff measure} on $\R^n$. In particular, if $\omega$ is a Borel subset of a $k$-dimensional euclidean space $E^k$ or a $k$-dimensional sphere $\S^k$  in $\R^n$, then $\mathcal{H}^k(\omega)$ coincides with the $k$-dimensional \emph{Lebesgue} measure of $\omega$ computed in $E^k$ or with the $k$-dimensional \emph{spherical Lebesgue} measure of $\omega$ computed in $\S^k$, respectively. 
\end{Def}

Let $K \subseteq \R^n$ be a compact, bounded, full-dimensional convex set with corresponding support function $\phi_K \colon \R^n \to \R$. Moreover, let 
\begin{align}\label{gaussmap}
g_K \colon \S^{n-1} \longrightarrow \mathcal{P}\left(\partial K\right),
\end{align}
where $\mathcal{P}\left(\partial K\right)$ denotes the power set of the boundary $\partial K$ of $K$, be the map given in the following way. In the case that $\phi_K$ is of class $C^2$, $g_K$ sends $u \in \S^{n-1}$ to the gradient $\nabla \phi_K(u)$. In general, the inverse $g^{-1}$ is what in the literature is called the \emph{Gauss} map, which assigns the outer unit normal vector $v_K(x)$ to an $x \in \partial_*K$, where $\partial_*K$ consists of all points in the boundary $\partial K$ of $K$ having a unique outer normal vector. In other words, we have that 
\[
g_K(u) = \left\{ m \in \R^n \, \big{|} \, \langle m,u\rangle = \phi_K(u) \text{ and } \langle m, v \rangle \geq \phi_K(v), \, \forall v \in \R^n\right \},
\]
for every $u \in \S^{n-1}$. This is also the map given by the Legendre--Fenchel duality (see Section \ref{section-duality}).
\begin{Def}\label{surfarea}
The \emph{surface area measure} $S_{n-1}(K, \cdot)$ associated to $K$ is the finite Borel measure on the unit sphere $\S^{n-1}$ defined by
\[
S_{n-1}(K,\omega) = \mathcal{H}^{n-1}\left(g_K(\omega)\right)
\]
for every Borel subset $\omega$ of $\S^{n-1}$. \\
In particular, for a polytope $P$ with unitary normal vectors $u_1, \dotsc, u_r$ at its facets $F_1, \dotsc, F_r$, respectively, the surface area measure of a Borel subset $\omega \subseteq \S^{n-1}$ is given by
\[
S_{n-1}(P, \omega) = \sum_{u_i \in \omega} \vol_{n-1}\left(F_i\right),
\] 
where $ \vol_{k}$ denotes the $k$-dimensional volume operator. In other words, we have that 
\[
S_{n-1}(P, \cdot) = \sum_{i=1}^r\vol_{n-1}\left(F_i\right) \delta_{u_i},
\]
where $\delta_{u_i}$ denotes the Dirac delta measure supported on $u_i \subset \S^{n-1}$ for all $i = 1, \dotsc, r$.
\end{Def}
\begin{exa}\label{enya}
Let $P$ be a polytope and let $\phi_P$ be its corresponding piecewise linear, concave support function. We get the formula 
\[
n\, \vol_n(P) = \sum_{i=1}^r\phi_P(u_i)\vol_{n-1}(F_i) = \int_{\S^{n-1}}\phi_P(u) S_{n-1}(P, u)
\]
for the volume of $P$. This formula can be generalized to any full dimensional bounded, compact, convex set in the following way: for every $n$-dimensional bounded, compact, convex set $K$, the volume of $K$ is given by 
\begin{align}\label{volumeformula1}
\vol_n(K) = \frac{1}{n}\int_{\S^{n-1}}\phi_K(u)S_{n-1}\left(K, u\right).
\end{align}
\end{exa}
As is described in \cite[Section 5]{BM}, one can generalize the surface area measure associated to a single convex to a collection of $n-1$ (not necessarily distinct) convex sets. This is the so called \emph{mixed surface area measure}. We denote by $\mathcal{K}_n$ the set of compact, bounded, full-dimensional convex sets in $\R^n$. The following is \cite[Theorem 5.1.7]{BM}.
\begin{theorem/Def}\label{mixedvolume2} There is a nonnegative symmetric function $\V \colon \left(\mathcal{K}_n\right)^n \to \R$, called the \emph{mixed volume}, such that for every natural number $\ell$ and for every non-negative real numbers $\lambda_1, \dotsc, \lambda_{\ell}$, the equation 
\[
\vol_n(\lambda_1 K_1 + \dotsb + \lambda_{\ell} K_{\ell}) = \sum_{i_1, \dotsc , i_n = 1}^{\ell} \lambda_{i_1}\dotsm \lambda_{i_n}\V \left(K_{i_1}, \dotsc , K_{i_n}\right),
\]
where the sum on the left hand side is the Minkowski sum of convex sets, is satisfied for any collection of convex bodies $K_1, \dotsc, K_{\ell} \in \mathcal{K}_n$. 

Furthermore, there is a symmetric map $S$ from $\left(\mathcal{K}_n\right)^{n-1}$ into the space of finite Borel measures on $\S^{n-1}$, called the \emph{mixed surface area measure}, such that for every natural number $\ell$ and for every non-negative real numbers $\lambda_1, \dotsc, \lambda_{\ell}$, the equation 
\[
S_{n-1}\left(\lambda_1 K_1 + \dotsb + \lambda_{\ell} K_{\ell}, \omega \right) = \sum_{i_1, \dotsc , i_{n-1} = 1}^{\ell} \lambda_{i_1}\dotsm \lambda_{i_{n-1}}S \left(K_{i_1}, \dotsc , K_{i_{n-1}}, \omega\right)
\]
is satisfied for $K_1, \dotsc, K_{\ell} \in \mathcal{K}_n$ and for every Borel subset $\omega \subseteq \mathbb{S}^{n-1}$. Moreover, for $K_1, \dotsc, K_n \in \mathcal{K}_n$, the mixed volume $\V\left(K_1, \dotsc,K_n\right)$ can be expressed in terms of the mixed surface area measure in the following way
\[
\V\left(K_1, \dotsc,K_n\right) = \frac{1}{n}\int_{\S^{n-1}}\phi_{K_1}(u)S\left(K_2, \dotsc,K_n, u \right).
\]
\end{theorem/Def}
We make the following remarks.
\begin{rem}
(1) Setting $K = K_1 = \dotsb = K_n$ we get 
\begin{align*}
\vol_n(K) & = \V(K, \dotsc, K) \\ & =  \frac{1}{n}\int_{\S^{n-1}}\phi_{K}(u)S\left(K, \dotsc,K, u \right) \\ & =  \frac{1}{n}\int_{\S^{n-1}}\phi_{K}(u)S_{n-1}\left(K, u \right)
\end{align*}
as in Equation (\ref{volumeformula1}).

\noindent (2) The mixed volume \enquote{$\V(\cdot)$} defined above is related to the Mixed Volume \enquote{$\MV(\cdot)$} from Equation \ref{def:mixedvolume} by the formula 
\[
\V\left(K_1, \dotsc,K_n\right) = \frac{1}{n!}\MV\left(K_1, \dotsc,K_n\right)
\]
for $K_1, \dotsc, K_n \in \mathcal{K}_n$. 
\end{rem}

Now, let $K \subseteq M_{\R}$ be a convex set. Then the support function $\phi_K$ of $K$ corresponds to a nef toric $b$-divisor $\D_K$. The following lemma follows from the definition of the surface area measure and Theorem \ref{the:volume}.
\begin{lemma}
Let us fix an identification $N_{\R} \simeq \R^n$. Let $\D_1, \dotsc, \D_{n}$ be a collection of nef toric $b$-divisors associated to convex sets $K_i \subseteq M_{\R} \simeq \R^n$ with corresponding support functions $\phi_i$, for all $i =1, \dotsc, n$. Then the mixed degree $\D_1 \dotsm \D_m$ is related to the mixed surface area measure $S_{n-1}(K_1, \dotsc, K_{n-1}, \cdot)$ by the formula 
\[
\D_1 \dotsm \D_{n} = (n-1)!\int_{\S^{n-1}}\phi_1(u) S(K_2, \dotsc, K_n,u).
\]
Moreover, by the symmetry of the mixed surface area measure, we have integral formulae  
\[
\D_1 \dotsm \D_{n} = (n-1)!\int_{\S^{n-1}}\phi_i(u) S(K_1, \dotsc, \widehat{K}_i,\dotsc, K_n,u)
\]
for all $i = 1, \dotsc, n$. 
\end{lemma}
\begin{rem}\label{rem:smoothcase}
Assuming some smoothness conditions on the support functions of the convex sets, one can compute integrals with respect to (mixed) surface area measure measures explicitly in terms of Lebesgue measures of determinants of Hessians of smooth functions (see \cite[Corollary 2.5.3]{BM} and the results in \cite[Section 5.3]{BM}).
\end{rem}

\section{Volumes and intersection numbers}\label{section-volume}

Throughout this section, $N$ will denote a lattice of rank $n$ and $M = N^{\vee}$ its dual lattice. All convex sets in $M_{\R}$ are assumed to be compact, bounded and full-dimensional. We fix two convex sets $K_1 \subseteq K_2$ in $M_{\R}$.

The goal of this section is to relate the correction terms of Definition \ref{correctionterm} with intersection numbers of toric $b$-divisors in the case that $K_2$ is a polytope.

Note that two related exposed faces $F_1 \sim F_2$ with $F_1 \leq K_1$ and $F_2 \leq K_2 $ are contained in parallel hyperplane sections (defined by the $v$ given in the proof of Theorem \ref{convexdecom}). The following is a key Lemma. 
\begin{lemma}\label{lemata}
Let $F_1, \, F_2 \subseteq \R^{d+1}$ be polytopes. Here, $d = \max\left\{\dim\left(F_1\right), \dim\left(F_2\right)\right\}$. Assume that $F_1 \subseteq \{x_{d+1} = 0\}$ and that $F_2 \subseteq \{x_{d+1} = 1\}$. Then the volume of the convex hull of $F_1, F_2$ is given by  
\[
\vol\left(\convhull\left(F_1, F_2\right)\right) = \frac{1}{d+1}\sum_{i=0}^d\MV\left(\underbrace{F_1, \dotsc, F_1}_{i \timess}, \underbrace{F_2, \dotsc, F_2}_{(d-i) \timess}\right).
\]
\end{lemma} 
\begin{proof}
We start with the following three claims: 

\noindent \emph{Claim 1:} Let $\lambda$ be any real number between $0$ and $1$. Then the slice of the convex hull of $F_1$ and $F_2$ at $x_{d+1} = \lambda$ is given by  
\[
\convhull\left( F_1, F_2\right) \cap \{x_{d+1} = \lambda \} = \lambda F_1 + (1-\lambda)F_2, 
\]
where the sum in the right hand side is the Minkowski sum of convex sets. 

\noindent \emph{Claim 2}: Let $\lambda$ be any real number between $0$ and $1$. Then it follows that the volume of the slice $\lambda F_1 + (1-\lambda)F_2 \subseteq \convhull\left( F_1, F_2\right)$ is given by   
\[
\vol\left(\lambda F_1 + (1-\lambda)F_2\right) = \sum_{i=0}^d \binom{d}{i}\lambda^i(1-\lambda)^{d-i}\MV\left(\underbrace{F_1, \dotsc, F_1}_{i \timess}, \underbrace{F_2, \dotsc, F_2}_{(d-i)\timess}\right).
\]
\noindent \emph{Claim 3}: Let $\lambda$ be as before and let $\ell, \, k$ be two non-negative integers with $k\leq \ell$. We define the number $\I(\ell,k)$ by   
\[
\I(\ell,k) \coloneqq \int_0^1\lambda^k (1-\lambda)^{\ell-k} \, \d\lambda.
\]
Then the formula
\[
\I(\ell,k) = \left((\ell+1)\binom{\ell}{k}\right)^{-1}
\]
holds true.

Now, Claim 1 is clear and Claim 2 is a standard result in convex geometry. We proceed to give a proof of Claim 3: integrating by parts, we get  
\begin{align*}
\I(\ell,k) &= \int_0^1\lambda^k(1-\lambda)^{\ell-k}\,\d\lambda \nonumber \\ &= \frac{\lambda^{k+1}(1-\lambda)^k}{k+1}\,\bigg|^1_0+ \int_0^1\frac{\lambda^{k+1}}{k+1}(\ell-k)(1-\lambda)^{\ell-k+1}\,\d\lambda \nonumber \\ &= \frac{\ell-k}{k+1}\;\I(\ell,k+1).
\end{align*}
Moreover the values for $k=\ell$ and for $k=0$ are given by 
\begin{align*}
\I(\ell,\ell) &= \int_0^1\lambda^{\ell}\,\d\lambda = \frac{\lambda^{\ell+1}}{\ell+1}\,\Big|_0^1 = \frac{1}{\ell+1}, \\
\I(\ell,0) &= \int_0^1(1-\lambda)^{\ell} \, \d\lambda = \frac{-(1-\lambda)^{\ell+1}}{\ell+1}\,\bigg|_0^1 = \frac{1}{\ell+1}.
\end{align*}
Hence, we get 
\begin{align*}
\I(\ell,\ell-1) &= \frac{1}{\ell} \cdot \frac{1}{\ell+1},\\
\I(\ell,\ell-2) &=\frac{2}{\ell-1}\cdot \frac{1}{\ell} \cdot \frac{1}{\ell+1}, \\  &\dotsc  \\  
\I(\ell,k) &= \left((\ell+1)\binom{\ell}{k}\right)^{-1},
\end{align*}
as we wanted to show.

Finally, note that Claim 1, Claim 2 and Claim 3 imply that   
\begin{align*}
\vol\left(\convhull\left(F_1, F_2\right)\right) &= \int_0^1\vol\left(\lambda F_1 + (1-\lambda) F_2\right)\,\d\lambda \\
&= \int_0^1\sum_{i=0}^d \binom{d}{i}\lambda^i(1-\lambda)^{d-i}\,\MV\left(\underbrace{F_1, \dotsc, F_1}_{i \timess}, \underbrace{F_2, \dotsc, F_2}_{(d-i)\timess}\right)\d\lambda \\
&= \sum_{i=0}^d\binom{d}{i}\,\I(d,i)\,\MV\left(F_1, \dotsc, F_1, F_2, \dotsc, F_2\right) \\ &= \frac{1}{d+1}\sum_{i=0}^{d}\MV\left(F_1, \dotsc, F_1, F_2, \dotsc, F_2\right),
\end{align*}
concluding the proof of the lemma.
\end{proof}
Consider the convex sets $K_1 \subseteq K_2$ with corresponding support functions $\phi_1, \phi_2$. Moreover,  let  $\Sigma = \Sigma_{K_2 \setminus K_1} \subseteq N_{\R}$ be a difference conical subdivision. 

We have the following theorem.
\begin{theorem}\label{ii}
Let notations be as above and assume that $K_2$ is a polytope. Then the functions $\phi_1$ and $\phi_2$ correspond respectively to nef toric $b$-divisors $\D_1$ and $\D_2$ on the normal fan of $K_2$, denoted by $\Sigma_{K_2}$. We can express the difference $\D_2^n - \D_1^n$ of degrees of the nef toric $b$-divisors $\D_1$ and $\D_2$ as a finite sum of correction terms 
\[ 
\D_2^n - \D_1^n = \sum_{F \overset{\exposed}{\leq} K_2} c_F,
\]
where the correction terms $c_F$ are related to top intersection numbers of toric $b$-divisors by
\begin{align*}
c_F & = \sum_{i=0}^{n-1}(n-1)!\int_{\relint(\sigma_F)\cap \S^{n-1}}\left(\phi_1(u)-\phi_2(u)\right)S(\underbrace{K_1, \dotsc, K_1}_{i \timess}, \underbrace{K_2, \dotsc, K_2}_{(n-1-i) \timess},u),
\end{align*}
where $S(\cdot )$ is the mixed surface area measure defined in the previous section.
In particular, if $K_1$ is also polyhedral, then $\Sigma$ is a real rational, polyhedral fan and we get 

\[ 
c_F = \sum_{i=0}^{n-1} \sum_{\overset{r \in \text{relint}(\sigma_F)}{r \in \Sigma(1)}}\left(\phi_1(r)-\phi_2(r)\right) \D_1^{n-1-i}\D_2^i\ D_r,
\]
where $D_r$ is the divisor corresponding to the ray $r \in \Sigma(1)$ and all the intersection products are done in $\Sigma$. 
 
\end{theorem}
\begin{proof}
The first and last statement of the theorem follow from Theorem \ref{convexdecom} and Definition \ref{correctionterm}. For the statement regarding the expression of the correction terms in terms of intersection numbers, we have   
\begin{align*}
\phi_2^n - \phi_1^n &= \left(\phi_2-\phi_1\right)\sum_{i=0}^{n-1}\phi_1^i\phi_2^{n-1-i} \\
&= \sum_{F \overset{\exposed}{\leq} K_2}\sum_{i=0}^{n-1}(n-1)!\int_{\relint(\sigma_F)\cap \S^{n-1}}\left(\phi_1( u)-\phi_2(u)\right)S(\underbrace{K_1, \dotsc, K_1}_{i \timess}, \underbrace{K_2, \dotsc, K_2}_{(n-1-i) \timess},u), 
\end{align*}
concluding the proof of the theorem.

\end{proof}
\begin{exa}\label{monja}
Consider the fan of $ \P^2$ and the nef toric $b$-divisors $\phi_1$ and $\phi_2$ given by the concave functions $\phi_1, \phi_2 \colon \R^2 \to \R$ defined by
\[
\phi_1(a,b) = \begin{cases}\frac{ab}{a+b} \, ,&  \hspace{1cm} a,b \in \R_{\geq 0}, \\ \min \{0,a,b\} \, ,& \hspace{1cm} \text{otherwise}. \end{cases}
\]
and 
\[
\phi_2(a,b) =  \min \{0,a,b\} 
\]
and consider the corresponding convex sets $K_1 \subseteq K_2$. Note that $K_2$ is the $2$-dimensional symplex and $K_1$ is the convex set from Example \ref{exa:folliation}. 
The only face of the simplex $K_2$ whose associated correction term is non-zero is the vertex $F_0$ in Figure \ref{2dd}. 
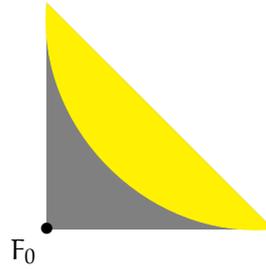
\begin{figure}[H]
\begin{center}
\begin{tikzpicture}[scale=3]
\filldraw[gray] (0,0)--(1,0)--(0,1)--cycle;
\filldraw[yellow] (0,1) to [bend right = 50](1,0) to [bend left= 0](0,1) --cycle;
\draw (0,0) node[below left]{$F_0$};
\draw (0,0) node{$\bullet$};
\end{tikzpicture}
\end{center}
\caption{Convex sets $K_1 \subseteq K_2$}\label{2dd}
\end{figure}
On the one hand, we can calculate the difference $\phi_2^2 - \phi_1^2$ as the difference of volumes of convex sets   
\[
c_{F_0} = \text{ full correction term } = \phi_2^2-\phi_1^2 = 2 \, \vol\left(K_2\right) - 2 \, \vol\left(K_1\right)= 1 - \frac{2}{3} = \frac{1}{3}.
\]
On the other hand, Theorem \ref{ii} tells us that we can compute the correction term $c_{F_0}$ by 
\[
c_{F_0} = \int_0^{\pi/2}\phi_1(\theta)S_1\left(K_1,\theta\right) = \frac{1}{3},
\]
where the last equality follows from a computation using Remark \ref{rem:smoothcase}. 

\end{exa} 

\printbibliography
%\bibliographystyle{amsalpha}
%\bibliography{bibliography}
\end{document}